\newif\ifORL
\ORLfalse 

\ifORL
	\documentclass[preprint,12pt]{elsarticle}
\else
	\documentclass{article}
\fi

\usepackage{amssymb}

\usepackage{amsfonts}
\usepackage{latexsym}
\usepackage{amsmath}
\usepackage{amssymb}
\usepackage{amsthm}
\usepackage{graphicx}
\usepackage{float}
\usepackage{caption}
\usepackage{subcaption}
\usepackage{color}
\usepackage[]{algorithm2e}
\usepackage{abstract}
\usepackage{algorithm2e}

\ifORL
\else
	\usepackage{fullpage}
\fi

\theoremstyle{definition}

\newtheorem{thm}{Theorem}

\newtheorem{prop}[thm]{Proposition}
\newtheorem{lemma}[thm]{Lemma}

\newtheorem{claim}[thm]{Claim}

\newtheorem{defn}[thm]{Definition}

\newcommand{\R}{\mathbb{R}}
\newcommand{\E}{\mathbb{E}}
\DeclareMathOperator{\Var}{Var}
\renewcommand{\P}{\mathcal{P}}

\newcounter{mynotes}
\ifORL
	\journal{Operations Research Letters}
\fi

\begin{document}

\ifORL
	\begin{frontmatter}
\fi


\title{Some lower bounds on sparse outer approximations of polytopes}


\ifORL
\else
	\author{Santanu S. Dey \and Andres Iroume \and Marco Molinaro}

	\maketitle
\fi

\begin{abstract}

Motivated by the need to better understand the properties of sparse cutting-planes used in mixed integer programming solvers, the paper~\cite{deygood} studied the idealized problem of how well a polytope is approximated by the use of sparse valid inequalities. As an extension to this work, we study the following ``less idealized" questions in this paper: (1) Are there integer programs, such that sparse inequalities do not approximate the integer hull well even when added to a linear programming relaxation?  (2) Are there polytopes, where the quality of approximation by sparse inequalities cannot be significantly improved by adding a budgeted number of arbitrary (possibly dense) valid inequalities? (3) Are there polytopes that are difficult to approximate under every rotation? (4) Are there polytopes that are difficult to approximate in all directions using sparse inequalities? We answer each of the above questions in the positive.
\end{abstract}

\ifORL

\end{frontmatter}
\fi




\section{Introduction}

The paper~\cite{deygood} studied how well one can expect to approximate polytopes  using valid inequalities that are sparse. The motivation for this study came from the usage of cutting-planes in integer programming (IP) solvers. In principle, facet-defining inequalities of the integer hull of a polytope can be dense, i.e. they can have non-zero coefficients for a high number of variables. In practice, however, most state-of-the-art IP solvers bias their cutting-plane selection towards the use of sparse inequalities. This is done, in part, to take advantage of the fact that linear programming solvers can harness sparsity well to obtain significant speedups.

The paper~\cite{deygood} shows that for polytopes with a polynomial number of vertices, sparse inequalities produce very good approximations of polytopes. However, when the number of vertices increase, the sparse inequalities do not provide a good approximation in general; in fact with high probability the quality of approximation is poor for random 0-1 polytopes with super polynomial number of vertices (see details in~\cite{deygood}).

However the study in~\cite{deygood} is very ``idealized" in the context of cutting-planes for IPs, since almost always some dense cutting-planes are used or one is interested in approximating the integer only only along certain directions.  In this paper, we consider some natural extensions to understand the properties of sparse inequalities under more ``realistic conditions":

\begin{enumerate}
\item All the results in the paper~\cite{deygood} deal with the case when we are attempting to approximate the integer hull using only sparse inequalities. However, in practice the LP relaxation may have dense inequalities.  Therefore we examine the following question: Are there integer programs, such that sparse inequalities do not approximate the integer hull well when added to a linear programming relaxation?

\item More generally, we may consider attempting to improve the approximation of a polytope by adding a few dense inequalities together with sparse inequalities. Therefore we examine the following question: Are there polytopes, where the quality of approximation by sparse inequalities cannot be significantly improved by adding polynomial (or even exponential) number of \emph{arbitrary} valid inequalities?

\item It is clear that the approximations of polytopes using sparse inequalities is not invariant under affine transformations (in particular rotations). This leaves open the possibility that a clever reformulation of the polytope of interest may vastly improve the approximation obtained by sparse cuts. Therefore a basic question in this direction: Are there polytopes that are difficult to approximate under \emph{every} rotation?

\item In optimization one is usually concerned with the feasible region in the direction of the objective function. Therefore we examine the following question: Are there polytopes that are difficult to approximate in \emph{almost all} directions using sparse inequalities?
\end{enumerate}

We are able to present examples that answer each of the above questions in the positive. This is perhaps not surprising: an indication that sparse inequalities do not always approximate integer hulls well even in the more realistic settings considered in this paper. Understanding when sparse inequalities are effective in all the above settings is an important research direction.

The rest of the paper is organized as follows. Section~\ref{sec:preliminary} collects all required preliminary definitions. In Section~\ref{sec:main} we formally present all the results. In Sections~\ref{lprelax}-\ref{sec:alldirection} we present proofs of the various results.


\section{Preliminaries}\label{sec:preliminary}

\subsection{Definitions}
For a natural number $n$, let $\left[n\right]$ denote the set $\left\{1,\ldots,n\right\}$ and, for non-negative integer $k\le n$ let ${\left[n\right]\choose k}$ denote the set of all subsets of $\left[n\right]$ with $k$ elements. For any $x\in \mathbb{R}^n$, let $||x||_1$ denote the $l_1$ norm of $x$ and $||x||$ or $||x||_2$ denote the $l_2$ norm of $x$.

An inequality $\alpha x \leq \beta$ is called \emph{$k$-sparse} if $\alpha$ has at most $k$ non-zero components. Given a polytope $P\subset \R^n$, $P^k$ is defined as the intersection of all $k$-sparse cuts valid for $P$ (as in \cite{deygood}), that is, the best outer-approximation obtained from $k$-sparse inequalities.

Given two polytopes $P, Q \subset \mathbb{R}^n$ such that $P \subseteq Q$ we consider the Hausdorff distance $d(P,Q)$ between them: $$d(P,Q):= \textup{max}_{x \in Q}\left(\textup{min}_{y \in P}||x - y||\right).$$
When $P, Q\subset \left[-1,1\right]^n$, we have that $d(P,Q)$ is upper bounded by $2\sqrt{n}$, the largest distance between two points in $\left[-1,1\right]^n$. In this case, if $d(P, Q) \propto \sqrt{n}$ the error of  approximation of $P$ by $Q$ is basically as large as it can be and smaller $d(P,Q)$ (for example constant or of the order of $\sqrt{\log n}$) will indicate better approximations.

Given a polytope $P \subseteq \R^n$ and a vector $c \in \R^n$, we define
\begin{align*}
gap^k_P(c) &= \max_{x\in P^k} cx - \max_{x\in P} cx,
\end{align*}
	namely the ``gap'' between $P^k$ and $P$ in direction $c$. We first note that $d(P, P^k)$ equals the worst directional gap between $P^k$ and $P$ (the proof is presented in Appendix A).
	
	\begin{lemma} \label{lemma:distGap}
		For every polytope $P \subseteq \R^n$, $d(P, P^k) = \max_{c:||c||=1} gap^k_P(c)$.
	\end{lemma}

For a set $\mathcal{D} = \left\{\alpha_1 x \leq \beta_1,\ldots, \alpha_d x \leq \beta_d\right\}$ of (possibly dense) valid inequalities for $P$, let $P^{k,\mathcal{D}}$ denote the outer-approximation obtained by adding all $k$-sparse cuts and the inequalities from $\mathcal{D}$:
\begin{align}
P^{k,\mathcal{D}} &= \left(   \bigcap_{i=1}^{d} \left\{  x\in \R^n : a_ix\le b_i   \right\}    \right)     \bigcap   P^k.
\label{pkc}
\end{align}
Since $P^{k,\mathcal{D}}\subseteq P^k$ we have that $d(P,P^{k,\mathcal{D}})\leq d(P,P^{k})$ for any set $\mathcal{D}$ of valid inequalities for $P$.

\subsection{Important Polytopes}

Throughout the paper, we will focus our attention on the polytopes $\mathcal{P}_{t,n} \subseteq [0, 1]^n$ defined as
\begin{align}
\mathcal{P}_{t,n} &=\left\{ x \in \left[0,1\right]^n:\sum_{i=1}^n x_i \leq t\right\}.
\label{goodset}
\end{align}
Notice that for $t = 1$ we obtain a simplex and for $t = n/2$ we obtain half of the hypercube. Moreover different values to $t$ yield very different properties regarding approximability using sparse inequalities, as discussed in~\cite{deygood}.
\begin{prop}\label{prop:pk}
	The following hold:
\begin{enumerate}
\item $d(\mathcal{P}_{1,n}, \mathcal{P}_{1,n}^k) =  \frac{\sqrt{n}}{k} - \frac{1}{\sqrt{n}}$.
\item $d(\mathcal{P}_{n/2,n}, \mathcal{P}_{n/2,n}^k) = \left\{\begin{array}{cl} \sqrt{n}/2  & \textup{if } k \leq n/2\\  \frac{n\sqrt{n}}{2k} - \frac{\sqrt{n}}{2} & \textup{if } k > n/2 \end{array} \right. $.
\end{enumerate}
\end{prop}

	We will also consider symmetrized versions of the polytopes $\P_{t,n}$. To define this symmetrization, for $x\in \mathbb{R}^n$ and $I\subset \left[n\right]$ let $x^I$ denote the vector obtained by switching the sign of the components of $x$ not in $I$:
\begin{align*}
x^I_i &=\left\{
\begin{array}{rl}
x_i &  \mbox{if } i \in I\\
-x_i & \mbox{if } i \notin I.\\
\end{array}
\right.
\end{align*}
More generally, for a set $P \subseteq \R^n$ we define $P^I =\left\{x^I \in \R^n: x \in P\right\}.$

\begin{defn}
For a polytope $P\subseteq \R_{+}^n$, we define its symmetrized version $\overline{P} =\mbox{conv}\left(\bigcup_{I\subseteq\left[n\right]} P^I\right).$
\end{defn}

Note that $\overline{\mathcal{P}_{1,n}}$ is the cross polytope in dimension $n$; more generally, we have the following external description of the symmetrized versions of $\P_{t,n}$ and $\P^k_{t,n}$ (proof presented in Appendix B).
\begin{lemma}\label{tildepk}
\begin{align}
\overline{\mathcal{P}_{t,n}} &=\left\{x \in \left[-1,1\right]^n:\ \forall I \subset \left[n\right],\ \sum_{i\in I}x_i - \sum_{i\in\left[n\right] \backslash I}x_i \leq t\right\}\\
\overline{\mathcal{P}_{t,n}}^k &=\left\{x \in \left[-1,1\right]^n:\ \forall I \in {\left[n\right]\choose k},\ \forall I^+,I^- \mbox{ partition of } I, \right.  \nonumber \\
		&~~~~~~~~~~~~~~~~~~~~~~~~~~~~~~~~~~~~~~~~~\qquad \left. \sum_{i\in I^+}x_i - \sum_{i\in I^-}x_i \leq t\right\}.
\end{align}
\end{lemma}

\section{Main results}\label{sec:main}


	In our first result (Section~\ref{lprelax}), we point out that in the worst case LP relaxations plus sparse inequalities provide a very weak approximation of the integer hull. 
	
\begin{thm}\label{thm:lp}
	For every even integer $n$ there is a polytope $Q_n \subseteq [0, 1]^n$ such that:
\begin{enumerate}
\item $\mathcal{P}_{n/2,n} = \textup{conv}(Q_n \cap \mathbb{Z}^n)$
\item $d(\mathcal{P}_{n/2,n}, (\mathcal{P}_{n/2,n})^k \cap Q_n) =  \Omega\left(   \sqrt{n}   \right)$ for all $k \leq n/2.$
\end{enumerate}
\end{thm}

In Section~\ref{strength} we consider the second question: How well does the approximation improve if we allowed a budgeted number of dense valid inequalities. Notice that for the polytope $\mathcal{P}_{\frac{n}{2},n}$, while Proposition \ref{prop:pk} gives that $d(\P_{\frac{n}{2},n}, \P_{\frac{n}{2},n}^k) \ge \Omega(\sqrt{n})$, adding exactly \emph{one} dense cut $\left(    ex\leq n/2    \right)$ to the $k$-sparse closure (even for $k = 1$) would yield the original polytope $\mathcal{P}_{\frac{n}{2},n}$.

	We consider instead the symmetrized polytope $\overline{\mathcal{P}_{\frac{n}{2},n}}$. Notice that while this polytope needs $2^n$ dense inequality to be described \emph{exactly}, it could be that a small number of dense inequalities, together with sparse cuts, is already enough to provide a good approximation; we observe that in higher dimensions valid cuts for $\overline{\mathcal{P}_{\frac{n}{2},n}}$ can actually cut off significant portions of $[-1,1]^n$ in \emph{multiple orthants}. We show, however, that in this even \emph{exponentially} many dense inequalities do not improve the approximation significantly.

\begin{thm} \label{thmbadset}
	Consider an even integer $n$ and the polytope $P = \overline{\P_{\frac{n}{2},n}}$. For any $k \leq n/100$ and any set $\mathcal{D}$ of valid inequalities for $P$ with $|\mathcal{D}| \le \exp\left( \frac{n}{600^2} \right)$, we have
$$ d\left(P,P^{k,\mathcal{D}} \right) \ge \frac{1}{6} \sqrt{n}.$$
\end{thm}
	In the proof of this theorem we use a probabilistic approach to count in how many orthants an inequality can significantly cut off the box $[-1,1]^n$.
	
In Section~\ref{sec:rotation} we consider the question of sparse approximation of a polytope when rotations are allowed. We show that again $\overline{\mathcal{P}_{n/2,n}}$ cannot be approximated using sparse inequalities after \emph{any} rotation is applied to it.

\begin{thm}\label{thm:rotation}
	Consider an even integer $n$ and the polytope $P = \overline{\P_{\frac{n}{2},n}}$. For every rotation $R : \R^n \rightarrow \R^n$ and $k  \leq\frac{{n}     }{200^3}$, we have
\begin{align*}
d   \big(  R(P), (R(P))^k     \big) &=\Omega( \sqrt{n}) .
\end{align*}
\end{thm}

	The proof of this theorem relies on the intuition given by Theorem \ref{thmbadset}: since $\overline{\P_{\frac{n}{2},n}}$ required exponentially many dense inequalities in order to be well approximated, no rotation is able to align all of them with the axis so that they can be captured by sparse inequalities.

	Finally, in Section \ref{sec:alldirection} we show that $\overline{\P_{\frac{n}{10},n}}$ and its $k$-sparse approximation have a large gap in almost every direction. 
\begin{thm}\label{thm:directional}
Let $n \geq 1000$ be an integer divisible by $10$ and consider the polytope $P = \overline{       \mathcal{P}_{n/10,n}     }$. If $C \in \R^n$ is a random direction uniformly distributed on the unit sphere, then for $k \leq \frac{n}{10}$ we have
\begin{align*}
			\Pr\left( gap^k_{P}(C)\geq \frac{\sqrt{n}}{20} \right) &\ge 1 - \frac{4}{n}.
\end{align*}
\end{thm}
	
	To prove this theorem we rely on the concentration of the value of Lipschitz functions on the sphere (actually we work on the simpler Gaussian space). 


\section{Strengthening of $LP$ relaxation by sparse inequalities}
\label{lprelax}

	We now present a short proof of Theorem \ref{thm:lp}. Consider the polytope
$$Q_{n} = \left\{x\in \left[0,1\right]^n:\sum_{i\in I} x_i \leq \frac{n}{2}\quad \forall I \in {\left[n\right]\choose \frac{n}{2}+1}\right\}.$$
It is straightforward to verify that $\mathcal{P}_{n/2,n} = \textup{conv}(Q_n \cap \mathbb{Z}^n)$.

From Part (2) of Proposition \ref{prop:pk}, $\mathcal{P}_{n/2,n}^k=\left[0,1\right]^n$ thus $Q_{n}\cap \mathcal{P}_{n/2,n}^k = Q_{n}$. Now $x=\frac{n}{n+2}e$ belongs to $Q_{n}$ and its projection onto $\mathcal{P}_{n/2,n}$ corresponds to $y=\frac{1}{2}e$. Therefore,
$$d  \left(       \mathcal{P}_{n/2,n},\mathcal{P}_{n/2,n}^k\cap Q_n          \right) = \frac{n - 2}{2n + 4}\sqrt{n}=\Omega(\sqrt{n}).$$ This concludes the proof of the theorem. 


\section{Strengthening by general dense cuts}\label{strength}

	Now we turn to the proof of Theorem \ref{thmbadset}. For that we will need Bernstein's concentration inequality (stated in a slightly weaker but more convenient form).

\begin{thm}[\cite{koltchinskii2011oracle}, Appendix A.2]
Let $X_1,X_2,\ldots,X_n$ be independent random variables such that $\E \left[ X_i \right]=0$ and $|X_i|\leq M\ \forall i$. Let $X=\sum_{i=1}^n X_i$ and $\sigma^2=\Var(X) \leq U$. Then:
\begin{align*}
\Pr\left(|X|>w\right) &\leq \exp\left(-\min\left\{  \frac{w^2}{4U},\frac{3w}{4M}  \right\}\right).
\end{align*}
\label{bern}
\end{thm}

	Notice that to prove the theorem it suffices to consider the case $k = \frac{n}{100}$, which is what we do. Recall that $P = \mathcal{P}_{n/2,n}$, consider any set $\mathcal{D}$ of valid inequalities for $P$ with $\|\mathcal{D}| \le \exp(\frac{n}{600^2})$; for convenience let $d=|\mathcal{D}|$. From Lemma \ref{tildepk} we know $P^k$ contains all the points in $\{-1,1\}^n$. Also note that for any $\bar{x} \in \{-1,1\}^n$ achieves the maximal distance in $P^k$ from $P$, namely $d (   P, P^k  ) = d(P, \bar{x}) = \frac{1}{2} \sqrt{n}$. We then consider a random such ``bad'' point $X$, namely $X$ is uniformly distributed in $\{-1,1\}^n$ (equivalently, the $X_i$'s are independent and uniformly distributed over $\left\{-1, 1 \right\}$). We will show that there exist an instantiation of the scaled random  $\frac{2 X}{3}$ which belongs to $P^{k, \mathcal{D}}$, which will then lower bound the distance $d(P, P^{k, \mathcal{D}})$ by $d(P, \frac{2 \bar{x}}{3}) = \frac{1}{6} \sqrt{n}$ (for some $\bar{x} \in \{-1,1\}^n$) and thus prove the result. 
	
	To achieve this, consider a single inequality $ax \le b$ from $\mathcal{D}$ (we assume without loss of generality that $\|a\|_1 = 1$). We claim that with probability more than $1 - \frac{1}{d}$, the point $\frac{2 X}{3}$ satisfies this inequality. By symmetry of $X$, we can assume without loss of generality that $a \ge 0$. To prove this, let $\bar{a}$ be the vector obtained by keeping the $k$ largest components of $a$ and zeroing out the other components (ties are broken arbitrarily), and let $\underline{a} = a - \bar{a}$. Since $\bar{a}x\leq b$ is a $k$-sparse valid inequality for $P$ and $X \in P^k$, we have that
\begin{eqnarray}
aX=\bar{a}X+\underline{a}X \le b + \underline{a}X. \label{eq1}
\end{eqnarray}

\begin{claim}
	$\Var(\underline{a} X) \le  \frac{b (n-k)}{k^2}$.
\end{claim}

\begin{proof}
\let\qed\relax
	Since $\Var(X_i)=1$ for all $i \in [n]$, we obtain that
\begin{eqnarray}
\Var(\underline{a}X) &=\sum_{i=1}^n \underline{a}_i^2 \Var(X_i) = ||\underline a ||^2. \label{eq2}
\end{eqnarray}
Note that the $k$th largest component of $a$ is at most $1/k$ (otherwise $\|a\|_1 > 1$), hence $\underline{a}_i X_i\leq \frac{1}{k}$ for all $i$, so we have
\begin{eqnarray}
||\underline{a}||^2 = \sum_{i=1}^n (\underline{a}_iX_i)^2 \leq \frac{1}{k} \sum_{i=1}^n \underline{a}_i X_i. \label{eq3}
\end{eqnarray}
Moreover, by comparing averages of the components of $\bar{a}$ and $\underline{a}$ and then using $\bar{a} e \le b$, we have that
\begin{eqnarray}
\sum_{i=1}^n \frac{\underline{a}_i}{n-k} \le \sum_{i=1}^n \frac{\bar{a}_i}{k}  \le  \frac{b}{k}. \label{eq4}
\end{eqnarray}
Now by using (\ref{eq2})-(\ref{eq4}), we obtain the bound $\Var(\underline{a} X) \le  \frac{b (n-k)}{k^2}$, thus concluding the proof.  \hfill $\diamond$
\end{proof}

Now using the fact that $|\underline{a}_i X_i| \le \frac{1}{k}$, $\E(\underline{a}X)=0$  and the above bound on $\Var(\underline{a}X)$, we obtain by an application of Bernstein's inequality (Theorem \ref{bern}) with $w= 30 b \frac{\sqrt{\log d}}{\sqrt{k}} $:

\begin{align}
	\Pr\left(      \underline{a}  X \ge 30 b \cdot \frac{\sqrt{\log d}}{\sqrt{k}}       \right)    \le      \exp \left(      - \min\left\{        \frac{30^2 b \cdot k \cdot \log d}{4(n-k)},                                                   \frac{30}{4}\cdot 3 b \cdot \sqrt{k \log d}       \right\}           \right). \label{eq:bernstein}   
\end{align}
	To upper bound the right-hand side of this expression, first we employ our assumption $d \le \exp(\frac{n}{600^2})$ and $k = \frac{n}{100}$ to obtain 
\begin{align*}
\sqrt{\log d} &\leq  \frac{1}{600}\sqrt{n} \leq  \frac{3\cdot 99}{30\cdot 10}\sqrt{n} = \frac{3}{30}\left(\frac{n - k}{k}\right)\sqrt{k}.
\end{align*}
	With this at hand, we have that the minimum in the right-hand side of \eqref{eq:bernstein} is achieved in the first term. Moreover, notice that $b \geq 1/2$: the point $p = (\frac{1}{2}, \ldots, \frac{1}{2})$ belongs to $P$ and hence $b \ge ap = \frac{1}{2} \|a\|_1 = 1/2$. Putting these observations together gives 
\begin{align*}
	\Pr\left(      \underline{a}  X \ge 30 b \cdot \frac{\sqrt{\log d}}{\sqrt{k}}       \right)  \le \exp \left(- \frac{30^2}{4\cdot 99}b \cdot \log d \right) < \exp(- \log d) = \frac{1}{d}.
\end{align*}

Then using \eqref{eq1} and the above inequality, we obtain that with probability more than $1-\frac{1}{d}$ we have
\begin{eqnarray}
aX &\le& b \left(1 + 30\frac{\sqrt{\log d}}{\sqrt{k}}   \right) \nonumber \\
\label{eq7} &= & b \left(1 + \frac{1}{2}\cdot 600\frac{\sqrt{\log d}}{\sqrt{n}}   \right) \le b\frac{3}{2},
\end{eqnarray}
where the first equality uses $k = \frac{n}{100}$ and the second inequality uses the assumption that $\sqrt{\log d} \leq  \frac{1}{600}\sqrt{n}$. Now note that (\ref{eq7}) implies that the point $\frac{2 X}{3}$ satisfies $ax \le b$ with probability more than $1 - \frac{1}{d}$.

	Since $|\mathcal{D}| = d$, we can then take a union bound over the above argument to get that with strictly positive probability $\frac{2X}{3}$ satisfies \emph{all} the inequalities in $\mathcal{D}$. Hence with strictly positive probability $\frac{2X}{3}$ belongs to $P^{k, \mathcal{D}}$ and in particular there is a point $\overline{x} \in \left\{ -1, 1\right\} ^n $ such that $\frac{2\overline{x}}{3} \in P^{k,\mathcal{D}}$.

	This gives the lower bound $d(P, P^{k,\mathcal{D}}) \ge d\left(         P, \frac{2\overline{x}}{3}       \right)$; now we lower bound the right-hand side. It is easy to see that the closest point in $\overline{P}$ to $2\overline{x}/3$ is $\overline{x}/2$, the projection onto $\overline{P}$. Since $||2\overline{x}/3-\overline{x}/2||= \frac{1}{6}||\overline{x}||$, we obtain that $d(\overline{P}, \overline{x}) \ge \frac{1}{6} \sqrt{n}$ which concludes the proof.


\section{Sparse approximation of rotations of a polytope}\label{sec:rotation}

	In this section we prove Theorem \ref{thm:rotation}; for that we need to recall some standard definitions from convex geometry.

	\begin{defn}
		Given a set $P \subseteq \R^n$:
		\begin{itemize}
			\item We say that $P$ is \emph{centrally symmetric} if $\forall x \in P:$ $-x\in P$.
			\item For any $\alpha \in \R$ we define the set $\alpha P := \left\{  \alpha   x : x\in P  \right\}$.
			\item The \emph{polar} of $P$ is the set $P^{\circ} = \left\{   z\in \R^n : zx\le 1 \ \forall x\in P\right\}$.
		\end{itemize}	
	\end{defn}



		


We also need the following classical result about approximating convex set by polytopes with few vertices (see for instance Lemma 4.10 of \cite{barvinok2013thrifty} and \cite{pisier1999volume})
\begin{thm}\label{lem:central}
For every centrally symmetric convex set $S \subseteq \R^k$, there is a polytope $S'$ with at most $(\frac{3}{\epsilon})^k$ {vertices} such that $S \subseteq S' \subseteq (1+\epsilon)S$
\end{thm}

By applying this result to the polar we obtain approximations with bounded number of \emph{facets} instead of vertices.

		\begin{lemma}\label{lem:centralmod}
			For every centrally symmetric conver set $C \subseteq \R^k$, there is a polytope $C'$ with at most $(\frac{3}{\epsilon})^k$ facets such that $C \subseteq C' \subseteq (1 + \epsilon) C$.
		\end{lemma}	
		
		\begin{proof}
		Consider the (centrally symmetric) convex set $\frac{1}{1+\epsilon} C^\circ$; applying the above result, we get $S$ with $(3/\epsilon)^k$ vertices and $\frac{1}{1+\epsilon} C^\circ \subseteq S \subseteq C^\circ$. Taking polars (and noticing that $(\lambda A)^\circ = (1/\lambda) A^\circ$), we get $C \subseteq S^\circ \subseteq (1+\epsilon) C$ and $S^\circ$ has at most $(3/\epsilon)^k$ facets. This concludes the proof.
		\end{proof}

The key ideas used in our proof of Theorem \ref{thmbadset} is twofold (recall that $P = \overline{\mathcal{P}_{n/2,n}}$):
\begin{enumerate}
\item Roughly speaking, $(RP)^k$ is the intersection of (rotations of) \emph{$k$-dimensional} polytopes. This allow us to use Lemma \ref{lem:centralmod} above (with $n$ set to $k$) to get a good approximation $H$ of $(RP)^k$ using fewer than $\exp(\frac{n}{600^2})$ inequalities. 


\item Then argue that $d(RP, (RP)^k) \approx d(RP, H) = \Omega( \sqrt{n})$ since $d(P, R^{-1}(H)) = \Omega( \sqrt{n})$ due to 
the number of facets of $H$ and Theorem \ref{thmbadset}.
\end{enumerate}

\begin{proof}[Proof of Theorem~\ref{thm:rotation}]
Note that it is sufficient to prove the result for $k = \frac{n}{200^3}$, which is what we do. To make the above ideas precise, observe that $(RP)^k = \bigcap_{K \in {[n] \choose k}} Q_K$, where $Q_K = RP + 0^{K}\times\R^{\bar{K}}$ (we use $\bar{K} := [n]\setminus K$). To approximate each $Q_K$, using Lemma~\ref{lem:centralmod}, let $h_K \subseteq \mathbb{R}^k$ be a polytope such that $\textup{proj}_{K}Q_K \subseteq h_K \subseteq (1+\epsilon) \textup{proj}_{K}Q_K$ and $h_K$ has at most $(3/\epsilon)^k$ facets. Let $H_k = h_K + 0^{K}\times\R^{\bar{K}}$; then $Q_K \subseteq H_K \subseteq (1+\epsilon) Q_K$ and $H_K$ has at most $(3/\epsilon)^k$ facets.

		Now notice that for convex sets $A, B$, we have $((1+\epsilon) A) \cap ((1+\epsilon) B) \subseteq (1+\epsilon) (A \cap B)$. This gives that if we look at the intersection $\bigcap_{K \in {[n] \choose k}} H_K$, we obtain
		\begin{align}
			(RP)^k &= \bigcap_{K \in {[n] \choose k}} Q_K \subseteq \bigcap_{K \in {[n] \choose k}} H_K \subseteq (1+\epsilon) \bigcap_{K \in {[n] \choose k}} Q_K \notag \\
		&= (1+\epsilon) (RP)^k.
		\label{eq}
		\end{align}
		
		 Notice $\bigcap_{K \in {[n] \choose k}} H_K$ has at most ${n \choose k} \left(\frac{3}{\epsilon}\right)^k \le \left(   \frac{en}{k}    \right)^k \left( \frac{3}{\epsilon}  \right)^k = \left(  \frac{3en}{k\epsilon}  \right)^k$ facets. Thus, setting $\epsilon = \frac{1}{10}$ we get 
\begin{eqnarray}
\left(  \frac{3en}{k\epsilon^{*}}  \right)^k &=& \left(  30 \cdot e\cdot  200^3  \right)^{\frac{n}{200^3}} \nonumber\\ 
&=& \left(\exp({ \log  (30 \cdot e \cdot 200^3)})\right)^{\frac{n}{200^3}} \nonumber \\
&=&\left(\exp \left(  \log  (30\cdot e \cdot 200^3) \cdot \frac{n}{200^3} \right)\right) \nonumber \\
& < & \label{eq10} \exp\left(\frac{n}{601^2} \right).
\end{eqnarray}
	Then define $H := \bigcap_{K \in {[n] \choose k}} H_K$, so that $(RP)^k \subseteq  H  \subseteq   (1+\epsilon) (RP)^k$. 

	In order to control the relationship between this multiplicative approximation and the distance $d(.,.)$, we introduce the set $C = R([-1,1]^n)$. Notice that by construction  $RP \subseteq H\cap C $.

\begin{claim}
	$d(RP, H \cap C) \geq \frac{1}{6}\sqrt{n}$
\end{claim}

\begin{proof}
\let\qed\relax
	Assume by contradiction that $d(RP, H \cap C) < \frac{1}{6}\sqrt{n}$. Then since distances between points and number of facets of a polytope are invariant under rotation, we obtain that  $d(P, R^{-1} (H \cap C)) < \frac{1}{6}\sqrt{n}$ where $R^{-1} (H \cap C)$ is defined using at most $\exp({n/(600)^2})$ inequalities (because $C$ has $2n$ facets, using (\ref{eq10}) $H$ has at most  $\exp({n/ (601)^2})$  and for sufficiently large $n$, $\exp({n/ (601)^2})+ 2n \leq \exp({n/ (600)^2})$). However notice that this contradicts the result of Theorem \ref{thmbadset}, since $k = \frac{\sqrt{n}}{100} \leq \frac{n}{100}$ and $ R^{-1} (H \cap C) $ is defined using at most $2^{n/(600)^2}$ inequalities.  \hfill $\diamond$
\end{proof}

But from (\ref{eq}) we have $(1+\epsilon) (RP)^k \cap C$ contains $H \cap C$, and hence
\begin{eqnarray}
\label{eq11} d(RP, (1+\epsilon) (RP)^k \cap C) \geq \frac{1}{6}\sqrt{n}.
\end{eqnarray}

\begin{claim}
	$d(RP, (RP)^k \cap C) \ge d(RP, (1+\epsilon) (RP)^k \cap C) - \epsilon \sqrt{n}$
\end{claim}

\begin{proof}
\let\qed\relax
			Take $\bar{x} \in (1+\epsilon) (RP)^k \cap C$ and $\bar{y} \in RP$ that achieve $d(\bar{x}, \bar{y}) = d((1+\epsilon) (RP)^k \cap C, RP)$. Look at the point $\frac{1}{1+\epsilon} \bar{x}$ and notice it belongs to $(RP)^k \cap C$; let $\tilde{y}$ be the point in $RP$ closest to $\frac{1}{1+\epsilon} \bar{x}$. Then since $\bar{y}$ is the point in $RP$ closest to $\bar{x}$,
			\begin{align*}
				d(RP, (1+\epsilon) (RP)^k \cap C) = d(\bar{x}, \bar{y}) \le d(\bar{x}, \tilde{y}).
			\end{align*}
		By triangle inequality, $d(\bar{x}, \tilde{y}) \le d(\frac{1}{1+\epsilon} \bar{x}, \tilde{y}) + d(\frac{1}{1+\epsilon} \bar{x}, \bar{x}) \le d(RP, (RP)^k \cap C) + d(\frac{1}{1+\epsilon} \bar{x}, \bar{x})$. To bound $d(\frac{1}{1+\epsilon} \bar{x}, \bar{x})$, notice it is equal to $\frac{\epsilon}{1+\epsilon} \|\bar{x}\|$; since $\bar{x}$ belongs to $C$, we can upper bound $\|\bar{x}\| \le \sqrt{n}$ (this is why we introduced the set $C$ in the argument). Putting these bounds together we obtain the result. \hfill $\diamond$
\end{proof}

Using (\ref{eq11}) and Claim 2 we obtain that
$d(RP, (RP)^k) \geq d(RP, (RP)^k \cap C ) \ge d(RP, (1+\epsilon) (RP)^k \cap C) - \epsilon \sqrt{n} \geq (\frac{1}{6} - \frac{1}{10})\sqrt{n}.$ This concludes the proof of the theorem.
		\end{proof}

		



\section{Lower bounds on approximation along most directions}\label{sec:alldirection}

	We now prove Theorem \ref{thm:directional}. The main tool we use in this section is concentration of Lipschitz functions on Gaussian spaces.
	
\begin{thm}[Inequality (1.6) of \cite{ledouxTalagrand}] \label{thm:concGauss}
	Let $G_1,G_2,\ldots,G_n$ be independent standard Gaussian random variables, and let $f:\R^n \rightarrow \R$ be an L-Lipschitz function, namely for all $x, x' \in \R^n$, $|f(x) - f(x')| \le L \cdot \|x - x'\|$. Then letting $Z=f(G_1,G_2,\ldots,G_n)$, for $t>0$ we have
	\begin{align*}
	\Pr\left( |Z - \E(Z)|   \geq t \right) &\le 2 \exp\left( -\frac{t^2}{2 L^2} \right) \\
	\end{align*}
\end{thm}
	
To prove Theorem \ref{thm:directional}, recall that $P = \overline{\mathcal{P}_{n/10,n}}$. Let $G = (G_1, G_2, \ldots, G_n)$ be a random vector whose components are independent standard Gaussians. It is well-known that $\frac{G}{\|G\|_2}$ is uniformly distributed in the sphere (see for instance \cite{ledouxTalagrand}, page 55). Notice that $gap^k_P(\cdot)$ is positive homogeneous, so $gap^k_P\left(\frac{G}{\|G\|}\right) = \frac{1}{\|G\|} \cdot gap_P^k(G)$.
	
Our first step is to lower bound $gap_P^k(G)$ with high probability, starting by lower bounding the maximization of $G$ over $P^k$.

\begin{claim}
	With probability at least $1 - \frac{1}{n}$, $\max_{x \in P^k} G x \ge 0.7n$.
\end{claim}

\begin{proof}
\let\qed\relax
		Since $k = \frac{n}{10}$, we have that $P^k = \left[-1,1\right]^n$ (Proposition \ref{tildepk}).
It then follows that
		\begin{align}
			\max_{x\in P^k} Gx = \sum_{i=1}^n |G_i|. \label{eq:lbGap1}
		\end{align}
		The random variables $|G_i|$ have \emph{folded normal} distribution \cite{foldednormal}, for which is known that $\E[|G_i|] = \sqrt{2/\pi} \ge 0.79$. Since the function $(x_1, \ldots, x_n) \mapsto \sum_{i=1}^n |x_i|$ is $\sqrt{n}$-Lipschitz, we can use Theorem \ref{thm:concGauss} to obtain the bound $$\Pr\left(\sum_{i=1}^n |G_i| < 0.7n\right) \le 2 \exp\left(-\frac{0.09^2 n}{2}\right) \le  \frac{1}{n},$$ where the last inequality holds if $n \ge 1000$. Equation \eqref{eq:lbGap1} then concludes the proof. \hfill $\diamond$
\end{proof}

	Next we upper bound the maximization of $G$ over $P$.

\begin{claim}	
	With probability at least $1 - \frac{2}{n}$, $\max_{x \in P} Gx \le 0.6n$.
\end{claim}

\begin{proof}
\let\qed\relax
 Letting $ext(P)$ denote the set of extreme points of $P$, notice that $\max_{x \in P} Gx = \max_{v \in ext(P)} Gv$, so it suffices to upper bound the latter. Also notice that the extreme points of $P$ are exactly the points in $\{-1,0,1\}^n$ with at most $\frac{n}{10}$ non-zero entries (Proposition \ref{tildepk}).
		
		Consider $v \in ext(P)$; we verify that $Gv \le 0.6n$ with probability at least $1 - 2 e^{-0.6 n}$. One way of seeing this, is by noticing that since $v$ has at most $\frac{n}{10}$ non-zero entries, $Gv = \sum_{i : v_i = 1} G_i + \sum_{i : v_i = -1} -G_i$ is a function of $G$ that has at most $\frac{n}{10}$ terms and is $\sqrt{\frac{n}{10}}$-Lipschitz, so Theorem \ref{thm:concGauss} gives
		\begin{align}
			\Pr\left(Gv > 0.6n\right) = \Pr\left(Gv - \E[Gv] > 0.6n\right) \le 2 e^{-0.6 n}, \label{eq:ubGap1}
		\end{align}
		and the result follows. (Another way to see this is to use that fact that $Gv$ is a centered Gaussian with variance at most $\frac{n}{10}$ and use a tail bound for the latter.)
		
		Now notice that $P$ has $\sum_{i=1}^{n/10} {n \choose i} 2^i \le \frac{n}{10} {n \choose n/10} 2^{n/10}$ extreme points. Since ${n \choose t} \le (\frac{e n}{t})^t$ for all $0 < t <n$, the number of extreme points of $P$ can be upper bounded by $$\exp\left(\ln\left(\frac{n}{10}\right) + \frac{n}{10} (\ln 10e + \ln 2)\right) \le \frac{2}{n} e^{0.6n},$$ where the last inequality uses $n \ge 30$.
		
		Then taking a union bound of \eqref{eq:ubGap1} over all extreme points of $P$ gives that with probability at least $1 - \frac{2}{n}$ for all $v \in ext(P)$ we have $Gv \le 0.6n$. This concludes the proof. \hfill $\diamond$
\end{proof}
	
	Finally, standard results give that $\|G\|_2 \le 2 \sqrt{n}$ with probability at least $1 - 2 e^{-0.5 n}$ (for instance, notice by Jensen's inequality $\E[\|G\|]^2 \le \E[\|G\|^2] = n$ and apply Theorem \ref{thm:concGauss} to $\|G\|$). Using the fact $n \ge 30$, we then get $\Pr(\|G\| \le 2 \sqrt{n}) \ge 1 - \frac{1}{n}$. Then taking a union bound over this event and the events $\max_{x \in P^k} G x \ge 0.7n$ and $\max_{x \in P} Gx \le 0.6n$ gives that with probability at least $1 - \frac{4}{n}$ we have $gap_P^k\big(\frac{G}{\|G\|}\big) = \frac{1}{\|G\|} \cdot gap_P^k(G) \ge \frac{\sqrt{n}}{20}$. This concludes the proof of Theorem \ref{thm:directional}.

\section*{Acknowledgments}
We thank Jon Lee for motivating some of the questions considered in this paper. Santanu S. Dey gratefully acknowledges the support by NSF under grant CCF-1415460.

\ifORL
\else
	\bibliographystyle{plain}
\fi
\bibliography{sparsebib}

\begin{thebibliography}{1}

\bibitem{barvinok2013thrifty}
Alexander Barvinok.
\newblock Thrifty approximations of convex bodies by polytopes.
\newblock {\em International Mathematics Research Notices}, page rnt078, 2013.

\bibitem{deygood}
Santanu~S Dey, Marco Molinaro, and Qianyi Wang.
\newblock How good are sparse cutting-planes?
\newblock In {\em Integer Programming and Combinatorial Optimization}, pages
  261--272. Springer, 2014.

\bibitem{koltchinskii2011oracle}
Vladimir Koltchinskii.
\newblock {\em Oracle Inequalities in Empirical Risk Minimization and Sparse
  Recovery Problems: Ecole d’Et{\'e} de Probabilit{\'e}s de Saint-Flour
  XXXVIII-2008}, volume 2033.
\newblock Springer, 2011.

\bibitem{ledouxTalagrand}
M.~Ledoux and M.~Talagrand.
\newblock {\em Probability in Banach Spaces: Isoperimetry and Processes}.
\newblock Springer, New York, 1991.

\bibitem{foldednormal}
F.~C. Leone, L.~S. Nelson, and R.~B. Nottingham.
\newblock The folded normal distribution.
\newblock {\em Technometrics}, 3(4):pp. 543--550, 1961.

\bibitem{pisier1999volume}
Gilles Pisier.
\newblock {\em The volume of convex bodies and Banach space geometry},
  volume~94.
\newblock Cambridge University Press, 1999.

\end{thebibliography}

\appendix
\section*{Appendix A}
	
\begin{proof}[Proof of Lemma \ref{lemma:distGap}]
		It is not difficult to see that for $P^k=P$, the lemma holds, since $d(P, P^k) = gap^k_P(c) = 0\ \forall c:||c||=1$. When, $P^k \neq P$, we have that $d(P, P^k)=d(x^0,y^0)>0$ is attained at $x^0\in ext(P^k)$ and $y^0\in P$, the orthogonal projection of $x^0$ onto $P$ (see \cite{deygood}). Thus, $y^0\in F=\left\{z\in \R^n:az=b\right\}\cap P$, a face of $P$ such that $a=(x^0-y^0)$, $b=(x^0-y^0)y^0$ and $P\subseteq \left\{z\in \R^n:az\le b\right\}$. Let $c=(x^0-y^0)/||x^0-y^0||$, we have: $\max_{x\in P} cx= cy^0$. On the other hand, $\max_{z \in P^k}cz=cx^0$, since otherwise, if $\exists \bar{x}\in P^k$ with $c\bar{x}>cx^0$, let $\bar{y}$ denote the orthogonal projection of $\bar{x}$ onto $\left\{z\in \R^n:az=b\right\}$. Then, for all $z\in P$ we have $d(\bar{x},z)\ge d(\bar{x},\bar{y})> d(x^0,y^0)$ (the last inequality follows from the fact that $c \bar{x} > c x^{0}$, $c\bar{y} = c y^0$ and $\bar{x} - \bar{y}/||\bar{x}-\bar{y}|| = c$), a contradiction. So, we obtain
\begin{align*}
d(P, P^k) &= ||x^0-y^0|| = c(x^0-y^0) \\
& = \max_{x\in P^k} cx -\max_{x\in P} cx= gap_P^k(c).
\end{align*}
Now, assume by contradiction that $\exists c'$ s.t. $gap_P^k(c')>gap_P^k(c)$ and $||c'||=1$. Let $x'\in P^k,y'\in P$ denote the points at which $gap_P^k(c')$ is attained. Using the definition of $d(P, P^k) $ and the relation between $c$ and $c'$ 
\begin{align*}
d(P, P^k) &\ge ||x'-y'|| = \frac{(x'-y')}{||x'-y'||}(x'-y') \\
& = \max_{c: ||c||=1} c(x'-y') \ge c'(x'-y')\\
&= gap^k_P(c') > gap^k_P(c) = d(P, P^k),
\end{align*}
a contradiction. Thus, we must have $d(P, P^k) = \max_{c:||c||=1} gap^k_P(c)$.
\end{proof}

\section*{Appendix B}

A polytope $P\subseteq \mathbb{R}^n_+$ is called \emph{down-monotone} if whenever $x\in P$ and $0\le y\le x$, we have $y\in P$. We begin with some preliminary results about the symmetrization we employ. 

\begin{lemma} \label{prelemma}
For a down-monotone polytope $P\subseteq \R_{+}^n$ we have $\overline{P}=\bigcup_{I\subseteq\left[n\right]} P^I$. 
\end{lemma}

\begin{proof}
It is sufficient to prove that the set $\bigcup_{I\subseteq\left[n\right]} P^I$ is convex. For that, consider $y^1,y^2\in \bigcup_{I\subseteq\left[n\right]} P^I$; by definition, let $x^1,x^2 \in P$ be such that there are sets $I_1, I_2$ giving $(x^1)^{I_1}=y^1$ and $(x^2)^{I_2} = y^2$. For any $\lambda\in\left[0,1\right]$, consider $y=\lambda y^1+ (1-\lambda) y^2$; we show $y \in \bigcup_{I \subseteq [n]} P^I$.

    By construction we have:
\begin{align*}
y_i=\left\{
\begin{array}{rl}
\lambda x^1_i+ (1-\lambda) x^2_i &  i\in I^1\cap I^2\\
\lambda x^1_i- (1-\lambda) x^2_i &  i\in I^1\backslash I^2\\
- \lambda x^1_i+ (1-\lambda) x^2_i &  i\in I^2\backslash I^1\\
- \lambda x^1_i- (1-\lambda) x^2_i &  i\in \left[n\right]\backslash I^1\cup I^2\\
\end{array}
\right.
\end{align*}
Now, let $\bar{I}=\left\{i \in \left[n\right]: y_i \ge 0 \right\}$. Then define $x :=  y^{\bar{I}}$, which is nonnegative by construction.
By non-negativity of the $x^i$'s, we have $|\lambda x^1_i- (1-\lambda) x^2_i | \le \lambda x^1_i+ (1-\lambda) x^2_i $ and  $ |-\lambda x^1_i + (1-\lambda) x^2_i | \le \lambda x^1_i + (1-\lambda) x^2_i $, thus $x \le x^1+ (1-\lambda) x^2\in P $. Since $P$ is down-monotone, we have that $x$ belongs to $P$. Since $y = x^{\bar{I}}$, this gives that $y$ belongs to $\bigcup_{I\subseteq\left[n\right]} P^I$, concluding the proof.
\end{proof}

\begin{lemma}\label{prelemma2}
	For a down-monotone polytope $P\subseteq \R_{+}^n$ we have $(\overline{P})^k=\overline{P^k}$.
\end{lemma}

\begin{proof}
	We break the proof into a couple of claims.

\begin{claim}\label{claim:down1} $(\overline{P})^k \cap \mathbb{R}^n_{+} = P^k= \overline{P^k} \cap \mathbb{R}^n_{+}.$
\end{claim}

\begin{proof}
\let\qed\relax
	For the first equality, notice that since $\overline{P}^k \supseteq P^k$ it suffices to prove $(\overline{P})^k \cap \mathbb{R}^n_{+} \subseteq P^k$. For any $x \in (\overline{P})^k \cap \R^n_+$ and $I \subseteq {\left[n\right]\choose k}$, there exists $y \in \overline{P}$ such that $y|_I = x|_I$. Moreover, using the fact that $x \ge 0$ and the symmetry in the definition of $\overline{P}$, there is one such $y$ which is non-negative, and hence $y \in P$. But again using $x|_I = y|_I$, we get that $x \in P^k$.
	
	For the second equality, since $P$ is down-monotone we have that $P^k$ is down monotone. Therefore, from Lemma \ref{prelemma} $\overline{P^k} = \bigcup_{I \subseteq [n]}(P^k)^I$, which implies $\overline{P^k} \cap \mathbb{R}^n_{+} = P^k$. \hfill $\diamond$
\end{proof}

\begin{claim}\label{claim:down2}
	Consider $z \in (\overline{P})^k$ and let $y = z^I$ for some $I \subseteq [n]$. Then $y \in (\overline{P})^k$.
\end{claim}

\begin{proof}
\let\qed\relax
 First note that it is straight forward to verify that if $\alpha x \leq b$ is a valid inequality for $\overline{P}$, then for every $I \subseteq [n]$ the inequality $a^I x \leq b$ is also a valid inequality for $\overline{P}$. Then the point $y$ must belong to $(\overline{P})^k$, since otherwise $y$ would be separated by some $k$-sparse cut $ax \le b$ and so $z$ would be separated by the $k$-sparse cut $a^I x \le b$. \hfill $\diamond$
\end{proof}

Now we conclude the proof of the lemma. For the direction $(\overline{P})^k \subseteq \overline{P^k}$, let $z \in (\overline{P})^k$ and let $I=\left\{i \in \left[n\right]: z_i \ge 0 \right\}$ and $x = z^I$. Then using Claim \ref{claim:down2} we get $x \in (\overline{P})^k \cap \mathbb{R}^n_{+}$. Thus by Claim \ref{claim:down1} we have $x\in P^k$ and hence $z \in \overline{P^k}$, concluding this part of the proof. For the direction $\overline{P^k} \subseteq (\overline{P})^k$, let $z \in \overline{P^k}$. Let $I=\left\{i \in \left[n\right]: z_i \ge 0 \right\}$ and $x = z^I$. The point $x \in \overline{P^k} \cap \mathbb{R}^n_{+}$. Thus, by Claim \ref{claim:down1} we have that $x \in (\overline{P})^k \cap \mathbb{R}^n_{+}$. However, by Claim \ref{claim:down2} we have that $z \in (\overline{P})^k$. This concludes the proof. 
\end{proof}

The next result together with Lemma~\ref{prelemma2} implies Lemma~\ref{tildepk}.

\begin{prop}
	Consider non-negative vectors $a^1, \dots, a^m \in \mathbb{R}^n_{+}$ and define the polyhedron $P = \{x \in \mathbb{R}^n_{+}\,|\, a^i x \leq b_i \ \forall i \in [m]\}$. Then $\overline{P} = \{x \,|\, (a^i)^I x \leq b_i \ \ \forall I \subseteq [n], \ \forall i \in [m] \}$.
\end{prop}
\begin{proof}

	$(\overline{P} \subseteq \{x \,|\, (a^i)^I x \leq b_i \ \forall I \subseteq [n], \ \forall i \in [m] \})$ Consider $z \in \overline{P}$ and define $I = \left\{i \in \left[n\right]: z_i \ge 0 \right\}$. Then $z^I \in \overline{P} \cap \mathbb{R}^n_{+}$ and thus $z^I \in P$ (from Lemma~\ref{prelemma}). Now observe that $(a^i)^I z = a_i z^I \leq b_i$ where the last inequality follows from that fact that $z^I \in P$. This concludes this part of the proof.

 $(\{x \,|\, (a^i)^I x \leq b_i \ \forall I \subseteq [n], \ \forall i \in [m] \} \subseteq \overline{P})$ Consider $z \in \{x \,|\, (a^i)^I x \leq b_i\ \forall I \subseteq [n], \ \forall i \in [m] \}$. Let $I = \left\{i \in \left[n\right]: z_i \ge 0 \right\}$. Then observe that $a_i z^I = (a^i)^I z \leq b_i$ for all $i \in [m]$ and $z^I \in \mathbb{R}^n_{+}$. Thus, $z^I \in P$ or equivalently, $z \in \overline{P}$. This concludes the proof.
\end{proof}

\end{document}